\documentclass[12pt,reqno]{amsart}
\usepackage{tikz}
\usetikzlibrary{shapes, backgrounds}

\usepackage[utf8]{inputenc}
\usepackage{fullpage, hyperref, url}
\usepackage{xargs}
\usepackage{amsmath,amsthm}
\usepackage{soul}
\theoremstyle{definition} 
\newtheorem{theorem}{Theorem}[section]

\newtheorem{proposition}{Proposition}[theorem]
\newtheorem{lemma}[theorem]{Lemma}

\newtheorem{example}[theorem]{Example}
\theoremstyle{definition}

\theoremstyle{remark}
\newtheorem*{remark}{Remark}

\usepackage{graphicx}

\newcommand{\eb}{$e$-power $b$-happy number}

\title{A tree approach to the happy function}
 \author{Eva G. Goedhart}
 \address[E.~G. Goedhart]{Department of Mathematics\\Bryn Mawr College\\
Bryn Mawr, PA 19101, USA} 
 \email{egoedhart@brynmawr.edu}
\author{Yusuf Gurtas}
\address[Y. Gurtas]{Department of Mathematics and Computer Science\\ Queensborough Community College, CUNY\\ 222-05 56th Avenue Bayside, NY 11364, USA}
\email{YGurtas@qcc.cuny.edu}
 \author{Pamela E. Harris}
 \address[P.~E.~Harris]{Department of Mathematical Sciences\\University of Wisconsin, Milwaukee\\
3200 N. Cramer Street\\
Milwaukee, WI 53227, USA} 
 \email{peharris@uwm.edu}
 
\date{\today}

\begin{document}

\maketitle

\begin{abstract}
    In this article, we present a method to construct $e$-power $b$-happy numbers of any height. 
    Using this method, we construct a tree that encodes these happy numbers, their heights, and their ancestry--relation to other happy numbers.
    For fixed power $e$ and base $b$, we consider happy numbers with at most $k$ digits and we  give a formula for the cardinality of the preimage of a single iteration of the happy function. 
    We show that these happy numbers arise naturally as children of a given vertex in the tree. 
    We conclude by applying this technique to $e$-power $b$-unhappy numbers of a given height.
\end{abstract}

\section{Introduction}

In 1994, Richard Guy \cite{guy} posed several questions about happy numbers. Happy number are integers that are defined in terms an operation on integers that, after repeated iterations of this operation, reaches one.
This operation can be expressed as a function as follows. Take a positive integer $n$ in decimal digit expansion as $n=\sum_{i=0}^m d_i(10)^i$ where $0\leq d_i\leq 9$ for all $1\leq i\leq m$, then define the function 
$S:\mathbb{Z}^+\to\mathbb{Z}^+$ by 
\begin{align}
\label{def:happy function}
S(n)=\sum_{i=0}^md_i^2,
\end{align}
which is referred to as the \emph{happy function}.
For a positive integer $k$, we write the  $k$-th iteration of this function as $S^k$ and let $S^1=S$. 
Then a number $n$ is said to be a \emph{happy number} if there exists a positive integer $\ell$ such that $S^\ell(n)=1$.  
When $n>1$, the smallest positive integer $\ell$ such that $S^\ell(n)=1$ is referred to as the \emph{height} of $n$. The height of $1$ is defined to be $0$, as defined in \cite{heights}.

Much work has been done to generalize and study happy numbers, for a survey see \cite{GHS2022}. 
This includes consideration of different positional bases including negative bases \cite{negbases, heights,fifth}. 
For a positive integer $n$ expressed in base $b$ as $n=\sum_{i=0}^m d_ib^i$, where $0\leq d_i\leq b-1$ for all $1\leq i\leq m$,  define the $e$-power, base $b$ happy function 
$S_{e,b}:\mathbb{Z}^+\to\mathbb{Z}^+$ by 
\begin{align}
\label{def:happy function gen}
S_{e,b}(n)=\sum_{i=0}^md_i^e.
\end{align}
Following the conventions of \cite{genhappy}, if there exists a positive integer $\ell$ such that $S_{e,b}^\ell(n)=1$, then $n$ is an \textit{$e$-power $b$-happy number}.
Known results in this area determine when positive integers are $e$-power $b$-happy and whether there exists arbitrarily long sequences of consecutive such happy numbers \cite{negbases,consec}. 
Further work has considered non-positional base representations of positive integers, such as the work of Carlson, Goedhart, and Harris utilizing factoradic base \cite{factoradic}, and the work of Trevi\~{n}o and Zhylinski utilizing fractional base representations \cite{trevino}. 

Another avenue of study of happy numbers stem from one of the questions posed by
Guy \cite[problem E34]{guy}, which can be rephrased as: {What is the smallest positive integer with a given height?}
In the case of positional number systems with base $b>1$, the smallest $e$-power $b$-happy numbers were provided by Grundman and Teeple in \cite{genhappy,heights,consec}. Additional results related to the height of happy numbers in positional bases can be found in \cite{onheights}.

For positive integers $e>1$ and $b>2$, there exist \emph{unhappy} numbers. That means that for some unhappy number $n$ there are  no values of $\ell$ satisfying $S_{e,b}^\ell(n)=1$.
Iteratively applying the happy function $S_{e,b}$ to an unhappy number $n$ is
eventually periodic but not equal to $1$. 
The sequence of numbers in the period of the unhappy number form what we call a \textit{cycle} of positive integers. 
Note that there may be more than one cycle, but every unhappy number reaches exactly one cycle. 
For an unhappy number, we will use the word \emph{height} to denote the smallest number of iterations of the happy function $S_{e,b}$ required to obtain one of the numbers in the cycle.

In this article, we present a graph theoretical approach
to study the iterations of the happy function, $S_{e,b}$ for any positive integers $e$ and $b$.
This perspective allows us to 
take advantage of the structure of trees and graphs with only one cycle to reveal some intrinsic properties that the function $S_{e,b}$ possesses. 
\begin{remark}
    Given that we work in a variety of bases, whenever we write an integer $x$ using a base  $b\neq 10$, we denote it by $x=\sum_{i=0}^{k}d_ib^i=(d_k,\cdots, d_1,d_0)_b$, where $0\leq d_i\leq b-1$ are the digits in base $b$.
\end{remark}

From those results, we give a formula for the number of children of a vertex which corresponds to the cardinality of the preimage of a single iteration of the happy function. 
For example, as we see in Proposition \ref{prop:T23} the number of children of a  $2$-power $3$-happy number $m$ with $k$ digits is
     \begin{align} \label{eq:ternary}
        \sum_{i=0}^{\lfloor m / 4 \rfloor}\binom{k}{i}\binom{k-i}{m-4i}.
   \end{align}  
The index $i$ in the above equation corresponds to the number of twos in the base $3$ representation of $m$ and $m-4i$ is the number of ones. 
Note that $\lfloor m / 4 \rfloor$ is the maximum number of twos allowed in the base $3$ representation of $m$.
The sum in equation \ref{eq:ternary} is specialization 
of our main result Theorem \ref{thm:children}, which gives the analogous count for all bases $b>1$ and exponents $e>1$.

This article is organized as follows. In Section \ref{sec:baseb}, we present an method that generates happy numbers of arbitrary height.
In Section \ref{sec:treeapproach}, we introduce the tree approach which allows us to visualize happy numbers as vertices in a tree.
In Section \ref{sec:cyclicnumbers}, we extend the tree approach to graphs with one cycle. This way we can visualize the unhappy numbers as vertices in a graph as well.

\section{A method to generate happy numbers of a given height}\label{sec:baseb}
In this section, we describe a process for generating happy numbers of a given height.
This process is one way to generate such numbers, and we remark that it is neither the only way nor the most optimal/efficient. Meaning, that the integers we create are often not the smallest integers that are $e$-power $b$-happy of a given height, nor is this process the fastest way to create integers with these properties.

To begin we utilize the greedy algorithm as follows.
\begin{enumerate}
    \item For a given integer $N>1$, set $N_1=N$ and find the largest positive integer $x_1$ such that $x_1(b-1)^e$ does not exceed $N_1$.
    \item Then let $N_2=N_1-x_1(b-1)^e$ and repeat the process by finding the the largest positive integer $x_2$ such that $x_2(b-2)^e$ does not exceed $N_2$.
    \item Iterate this process and output \[v=\{\underbrace{b-1,\ldots,b-1}_{x_1},\underbrace{b-2,\ldots,b-2}_{x_2},\ldots,\underbrace{1,\ldots,1}_{x_{b-1}}\},\]
    such that \[x_1(b-1)^e+x_2(b-2)^e+\cdots+x_{b-1}(1)^e=N.\]
\end{enumerate}

We now illustrate this implementation.  
\begin{example}\label{ex:1}
Let $N=586$ in base 10, choose base $b=16$, and exponent $e=2$. 
The above steps produces the set $v=\{15,15,11,3,2,1,1\}$
which satisfies
\begin{align*}
15^2+15^2+11^2+3^2+2^2+1^2+1^2=586.
\end{align*}
We use the numbers in the set $v$ to construct a base 16 number that
produces the smallest base 10 number which utilizes
the elements of $v$ as its digits.
That is, we take the numbers in $v$ as digits in base $b=16$ and we order them from smallest to largest. 
Observe that
\begin{align*}
  1\cdot (16)^6+ 1\cdot (16)^5+ 2\cdot (16)^4+3\cdot (16)^3+11\cdot (16)^2+ 15\cdot (16)^1+15\cdot (16)^0=17972223,
\end{align*}
which is $17972223=(1123BFF)_{16}$.
Note that 
\[S_{2,16}((1123BFF)_{16})=586.\]
\end{example}
We remark that in Example \ref{ex:1}, $17972223=(1123BFF)_{16}$ is not necessarily the smallest possible integer which satisfies  $S_{2,16}((1123BFF)_{16})=586$. 
A thorough search reveals that 
$24046=(5DEE)_{16}$ is the smallest possible number satisfying  $S_{2,16}((5DEE)_{16})=586$.

We now state our process formally:

\begin{center}
\textbf{Algorithm 1:} The greedy process\\
\noindent\fbox{%
\parbox{.6\textwidth}{
\begin{itemize}
\item[] \textbf{input}  $\{ N>1,b>1,e>1 $ ; integers $\}$ 
\item[]  \textbf{var} $ \{ i,j,k,t,v,w,R ;$ integers $\} $
\item[] $ v:=[\ \ ] ; R:= N $ 
\item[]  \textbf{for} j   \textbf{while}  $R>0$ \textbf{do}
\item[]   \hspace{1cm} \textbf{for} i \textbf{while} $ i^{e}\leq R $ \textbf{and} $ i<b $ \textbf{do}
\item[]  \hspace{1in} i := i +1
\item[] \hspace{1cm} $ t:=i-1 $
\item[]  \hspace{1cm} $ k:= \left \lfloor  R / t^{e} \right  \rfloor $
\item[]  \hspace{1cm} $ w:=[t \ldots t] $ \hspace{1cm}  \# length is $ k $
\item[]  \hspace{1cm} $ v:= Extend ( v,w) $
	\item[] \hspace{1cm} $ R:= R- kt^{e}$
		\item[] \hspace{1cm} $ j:= j+1$
	\item[]  \textbf{output} $ v $
\end{itemize}
}
}
\end{center}
Observe that \textbf{Algorithm 1} takes a given integer $N>1$ as the input, as well as an integer base, $b$, and exponent, $e$, and outputs a set $v=\{a_i\}$ with $0\leq a_{i}<b$ such that $N=\sum_{i=1}a_{i}^e$.
We then  
take the output $v$ of \textbf{Algorithm~1} to produce the smallest base 10 number utilizing the elements of $v$ as its digits.

Next, we describe how to utilize \textbf{Algorithm 1} to construct an $e$-power $b$-happy number of height $h>1$. 

\begin{enumerate}
    \item Fix a base $b>1$, power $e>1$, and choose height $h>1$.
    \item Let $N>1$ be a $e$-power $b$-happy number of height $1$.
    \item Apply \textbf{Algorithm 1} to $N$ with output $v$.
    \item Use $v$ to find $x$, 
    which is the smallest base 10 number that utilizes
the elements of $v$ as its digits and $S_{e,b}(x)=N$.
    \item Set $N_1=x$ and repeat the process for $N_1$ an additional $h-2$ times.
    \item Return the integer $N_{h-1}$.
\end{enumerate}

Note $N_{h-1}$ is a $e$-power $b$-happy number of height $h$. 
Moreover, one choice for an $e$-power $b$-happy number is $N=b$ in its base 10 representation. Since $N=b$ has the base $b$ representation $(1)_b$, one can see that it satisfies $S_{e,b}((1)_b)=1^e=1$.

We now provide a example.

\begin{example} 
Let $b=30$, $e=2$, and $h=3$. 
Set $N=30$ in base 10.
The process generates the sequence
$$ 1,30,965,838259,$$
implying $S_{2,30}(838259)=965$, $S_{2,30}(965)=30$, and finally $S_{2,30}(30)=1$. 
This shows that $838259$ in base 10 is a $2$-power $3$-happy number of height $3$.
\end{example}

As the previous example illustrates, the numbers generated can get big quite quickly.  
However, by trial and error, one can find $797$ in base 10, which is a $2$-power $30$-happy number of height $3$.
This shows that this method does not necessarily find the smallest happy number with a given height.

 \section{Tree Approach For Happy Numbers}\label{sec:treeapproach}
In order to visualize the sequences of integers arising from our construction of height $h$ happy numbers, we assemble the sequences described in the previous section into a tree. 

Fix base $b$ and power $e$. We construct an infinite tree (growing upward) whose root is labeled 1, 
and children are labeled with $e$-power $b$-happy numbers. 
If $n$ is an $e$-power $b$-happy number of height $h$, then $n$ is the label of a vertex which is a distance of $h$ away from the root of the tree. 
Since we are on a tree, there is a unique path connecting the vertex labeled by $n$ and the root of the tree.
We denote the infinite tree constructed in this fashion by $\mathcal{T}_{e,b}$.

For example, if $b=30$ and $e=2$, the first two levels of the tree $\mathcal{T}_{2,30}$ are illustrated in Figure \ref{fig:30}. We remark that the right most vertex labeled by ``$\cdots$'' denotes the fact that 
there are infinitely many happy numbers of every height.
This is easy to see because the happy numbers of height 1 are those that have a representation $(10)_b,(100)_b,\ldots $. 
Happy numbers with height two are the numbers that return a height one happy number upon a single application of the happy function.
Iteratively, a happy number with height $h$, is a number that returns a height $h-1$ happy number upon a single application of the happy function. 

\begin{remark}
Whenever an \eb~$n$ has height $h$, we say that the vertex in $\mathcal{T}_{e,b}$ labeled by $n$ lies at level $h$. Moreover, we will abuse notation and refer to a vertex in the tree by its label $n$.
\end{remark}

\begin{figure}[h]
  \centering
\includegraphics[width=16cm]{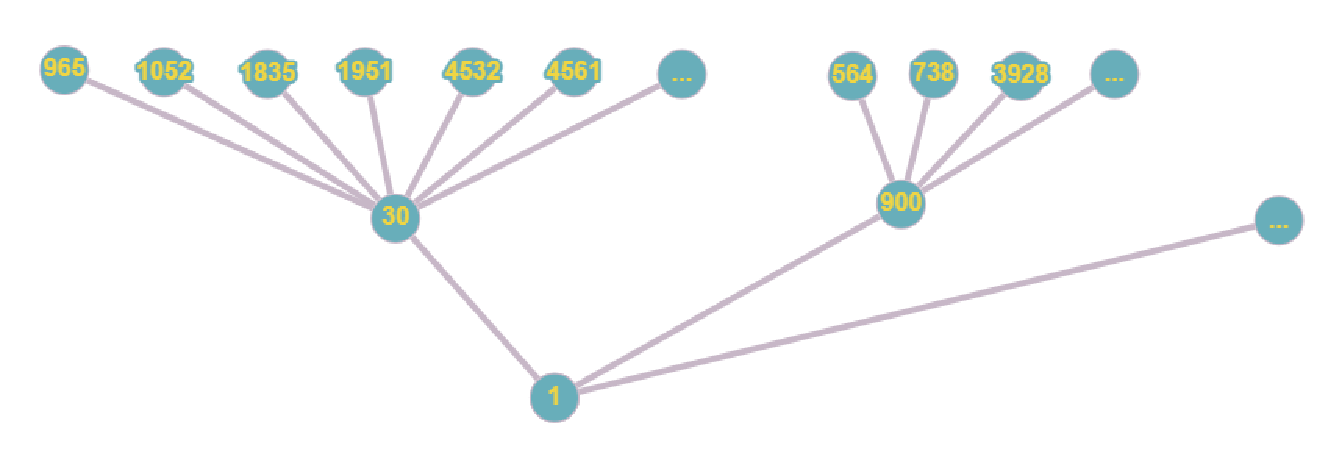}
 \caption{The first two levels of the tree $\mathcal{T}_{2,30}$.}  \label{fig:30}
\end{figure}

\begin{example}
The following sequence illustrates a $2$-power $30$-happy number of height 61, which utilizes relatively small values:

\noindent\begin{align} \label{seq:2005}
&1, 30, 965, 797, 356, 490, 637, 441, 630, 985, 389, 520, 554, 173, 73, 243, 1241, 629, 85,\\\nonumber
&272, 136, 306, 1037, 884, 850, 821, 445, 81, 9, 962, 881, 505, 261, 976, 740, 676, 26, 905, \\\nonumber& 269, 400, 20, 122, 41, 125, 305, 137, 1090, 829, 820, 386, 169, 365, 433, 522, 291, 1301,\\\nonumber& 776, 326, 1097, 509, 665, 2005.
\end{align}
Hence, $S_{30,2}^{61}(2005)=1$. 
The sequence \eqref{seq:2005} indicates that 2005 is a vertex in $\mathcal{T}_{2,30}$ at level 61. 
\end{example}

We make the following observations: 
\begin{enumerate}
    \item Given that we may introduce any number of zeros as digits, there are infinitely \eb s of a fixed height.
    \item The tree $\mathcal{T}_{e,b}$ is not an $m$-ary tree for any positive integer $m$. 
    \item The tree $\mathcal{T}_{e,b}$ has no leaves. 
\end{enumerate}

We order the (infinite list of) children of each vertex 
in increasing order from left to right according to their base 10 representation.
Then, based on this ordering, label the edges connecting the parent to the child using the set of natural numbers in order.
For example, in Figure \ref{fig:30} the label of 965 is 1 and the label of 1052 is 2.

We define the \emph{value of a vertex} 
as the sum of the labels of the edges on  the unique path connecting it to the root. 
For example, extending Figure \ref{fig:30}, the value of the vertex 2005 is 113, which is obtained from the sum of the edges labels traversed from the vertex 2005 to the root of the tree. These path labels are: 
\begin{align*}
1,1,2,1,2,2,1,2,1,1,2,3,1,1,2,1,1,1,4,1,2,3,4,4,5,1,1,1,3,\\2,1,1,3,2,3,1,3,1,2,1,2,1,1,3,1,3,2,2,2,1,3,3,2,1,3,1,1,2,1,2,1.
\end{align*}
The sequence of edge labels above is the unique \emph{address} of the vertex 2005 within the tree. 
The first 1 corresponds to the vertex labeled 30, the second 1 corresponds to the vertex labeled 965, and so on.

Note that the length of an address uniquely identifies the height of the happy number, yet it does not identify the number itself. 
We wonder if the value of a vertex can help identify the address of small happy numbers of a given height. 
 
 \subsection{Finite trees}
 As we saw in the previous section, the infinitude of happy numbers of a given height yields too large of playground. 
 Hence we fix the number of digits allowed in a base $b$ representation of \eb s to prune the infinite tree down to a finite tree. 
 Throughout, we let $k\geq 1$ denote the maximum number of digits, and we let $\mathcal{T}_{e,b}^k$ denote the finite tree arising from pruning $\mathcal{T}_{e,b}$ to only those numbers with at most $k$ digits in their base $b$ representation.

\subsubsection{Binary}
 If the base $b=2$, then every number is happy for all powers $e>1$. 
 If we restrict to numbers with at most $k$ digits, then we get a tree of happy numbers with $2^{k}-1$ vertices. 
 Figure \ref{fig:2-5} illustrates the finite tree $\mathcal{T}_{e,2}^5$ which has 31 vertices and three levels.

 \begin{proposition}
 Fix $k>1$ and consider the vertex labeled $m$ in the finite tree $\mathcal{T}_{e,2}^k$.
 \begin{enumerate}
     \item If $m=1$, then $m$ has $k-1$ children. 
     \item If $1<m\le k$, then $m$ has $\binom{k}{m}$ children.
     \item If $m>k$, then $m$ has no children.
 \end{enumerate}
 \end{proposition}
 \begin{proof}
 In base 2, all numbers with a single nonzero digit are happy and have height 1. 
 Limiting to those with at most $k$ digits ensures that the only happy numbers with height 1 are: $(10)_2,\ldots, (1\underbrace{00\cdots0}_{k-1})_2$. There are $k-1$ such numbers. Thus the root has exactly $k-1$ children.

 The children of a vertex $1<m\le k$ are the numbers represented by bit strings of length $k$ with exactly $m$ bits equal to 1. There are $\binom{k}{m}$ many options for these bits, which establishes that there are $\binom{k}{m}$ children.

Assume $m>k$. Using the greedy algorithm: setting $N_1=m$ we find the largest positive integer $x_1$ such
that $x_1(b-1)^e=x_1(2-1)^e=x_1$ does not exceed 
$m$. This implies that $x_1=m$ and 
Algorithm 1 outputs $v=\{\underbrace{1,1,\ldots,1}_{m}\}$.
Then $(\underbrace{1,\ldots,1}_{m})_2$ which has $m$ digits, and since $m>k$, this will not be a child of the vertex labeled $m$ in $\mathcal{T}_{e,2}^k$.
 \end{proof}

\begin{figure}[h]
  \centering
  \includegraphics[width=16cm]{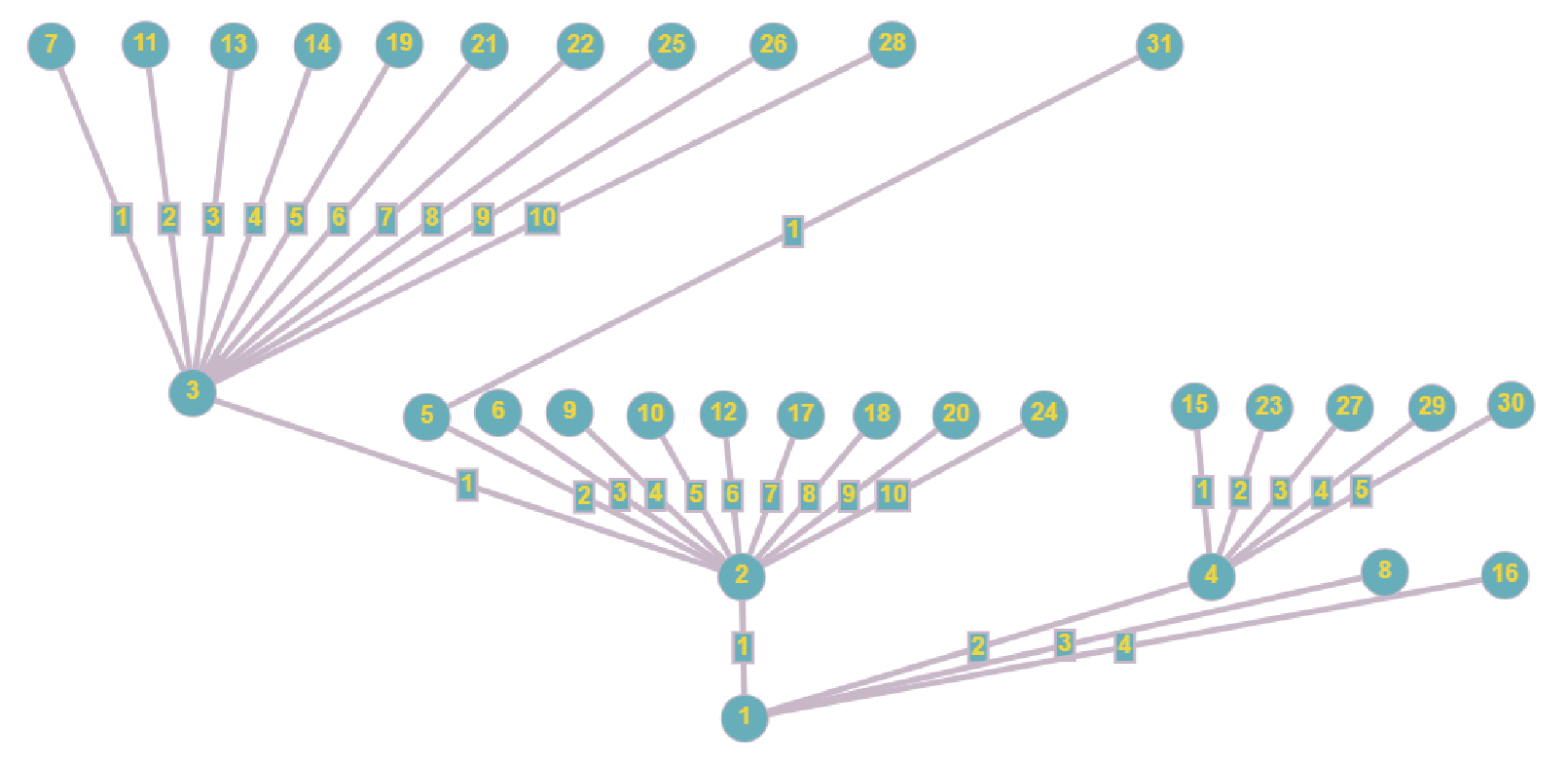}
  \caption{The three levels of the finite tree $\mathcal{T}_{e,2}^5$.}  \label{fig:2-5}
\end{figure}

\begin{figure}[h]
  \centering
  \includegraphics[width=15cm]{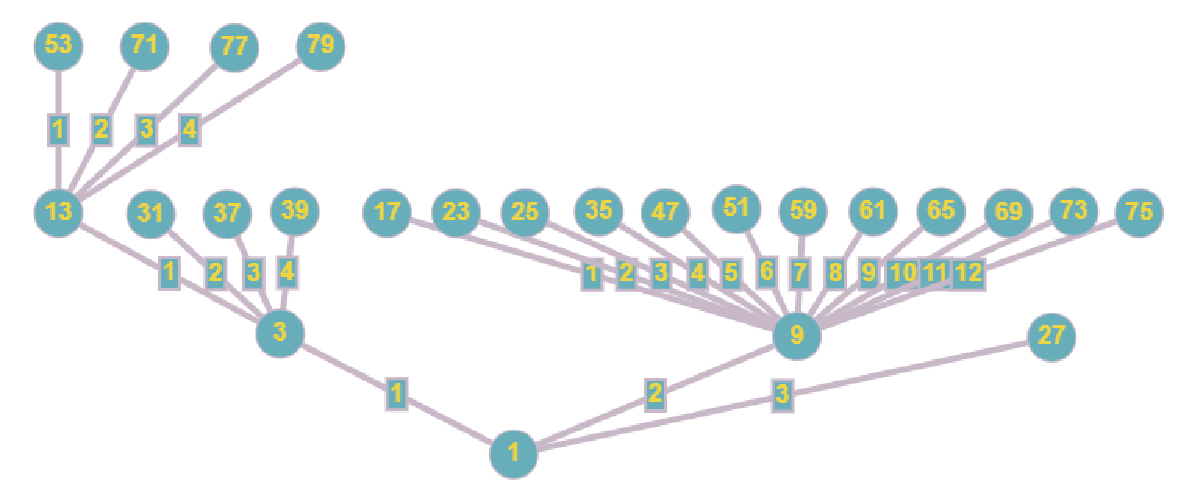}
  \caption{The three levels of the finite tree $\mathcal{T}_{2,3}^4$.}  \label{fig:3-2-4}
\end{figure}

\subsubsection{Ternary} Fix the base $b=3$. 

\begin{lemma}
Given $k>2$ and $e>1$, the vertex labeled $3$ in $\mathcal{T}_{e,3}^k$ has $\binom{k}{3}$ children.
\end{lemma}
\begin{proof}
    Using the greedy algorithm: setting $N_1=3$ we find the largest positive integer $x_1$ such
that $x_1(b-1)^e=x_1(3-1)^e=x_1\cdot 2^e$ does not exceed 
$3$. 
This implies that $x_1=0$ for all $e>1$. 
This implies that any child of 3 has no twos in any of its digits.
Thus they must only have zeros and ones as their digits. 
Since $3=(10)_3$ is an $e$-power $3$-happy number for all $e>1$, and a child of 3 must return to 3 upon one application of the happy function, it must be that any child of 3 can only have three ones in their ternary expansion.
Since these children must have exactly 3 digits equal to 1 and can have at most $k$ digits, there are $\binom{k}{3}$ many children.
\end{proof}

See Figure \ref{fig:3-2-4} for an illustration of the finite tree $\mathcal{T}_{2,3}^4$. 
Next we fix $e=2$ and give a count for the number of children of vertices in $\mathcal{T}_{2,3}^k$.
\begin{proposition}\label{prop:T23}
 
   The vertex $m>3$ in $\mathcal{T}_{2,3}^k$ has 
   \begin{align} \label{number-of-children-happy-m=4u+v}
\sum_{j=0}^{\lfloor\frac{m}{4}\rfloor}\binom{k}{j}\binom{k-j}{m-4j}
   \end{align}  
   children.
\end{proposition} 

\begin{proof}
Let $m>3$ be a vertex in $\mathcal{T}_{2,3}^k$. We  want to enumerate the children of $m$. 
Equivalently, we can construct all positive integers whose ternary representation has at most $k$ digits, and satisfy that a single application of the happy function returns $m$.
Namely, $x=(d_{k-1},d_{k-2},\ldots,d_0)$ is a child of $m$ if and only if all $0\leq d_i\leq 2$ and $\sum_{i=0}^{k-1}d_i^2=m$.
Let $j$ denote the number of digits in $x$ that are equal to 2. 
Note ${0\leq j\leq \lfloor\frac{m}{4}\rfloor}$, where $\lfloor\frac{m}{4}\rfloor$ is the largest integer such that $2^2\lfloor\frac{m}{4}\rfloor\leq m$.
Then there must be $m-j\cdot (2)^2=m-4j$ digits in $x$ equal to 1.
All remaining $k-(j+(m-4j))$ digits of $x$ will be zero.
This ensures that $S_{2,3}(x)=m$.

Since $x$ has at most $k$ digits then the number of ways to places the $j$ digits equal to two, and the $m-4j$ digits equal to one is $\binom{k}{j}\binom{k-j}{m-4j}$. 
Once the twos and ones have been placed the rest of the digits are zero and uniquely determined. {Taking the sum over $0\leq j\leq \lfloor\frac{m}{4}\rfloor$ yields the desired result.}
\end{proof}

\begin{proposition}\label{prop:Te3}
Fix $e>1$. The vertex $m>3$ in $\mathcal{T}_{e,3}^k$ has 
\begin{align} \label{base-3-children-of-m}
     \sum_{j=0}^{\lfloor\frac{m}{2^e}\rfloor}\binom{k}{j} \binom{k-j}{m-j\cdot2^e}
\end{align}  
   children.
\end{proposition}

\begin{proof}
This proof is analogous to that of Proposition \ref{prop:T23}. It suffices to replace $4$ by $2^e$ because each selection of a digit being two, removes $2^e$ choices for possible digits that equal one. Note the maximum possible number of digits that can be selected as a two is given by $\lfloor\frac{m}{2^e}\rfloor$. This completes the proof.
\end{proof}

\begin{example}
Figure \ref{fig:3-3-4} illustrates the tree $\mathcal{T}_{3,3}^4$.  Note that, $63=(2100)_3=2\cdot 3^3+1\cdot 3^2$ and $9=2^3+1^3$, hence $63$ is a child of $9$. Moreover, by Proposition \ref{prop:Te3}, $m=9$ has 
\[\sum_{j=0}^{\lfloor\frac{9}{2^3}\rfloor}\binom{4}{j}\binom{4-j}{9-8j}=\binom{4}{0}\binom{4}{9}+\binom{4}{1}\binom{3}{1}=4\cdot 0+4\cdot 3=12\]
children, as depicted in Figure \ref{fig:3-3-4}.
\end{example}

\begin{figure}[h]
  \centering
  \includegraphics[width=15cm]{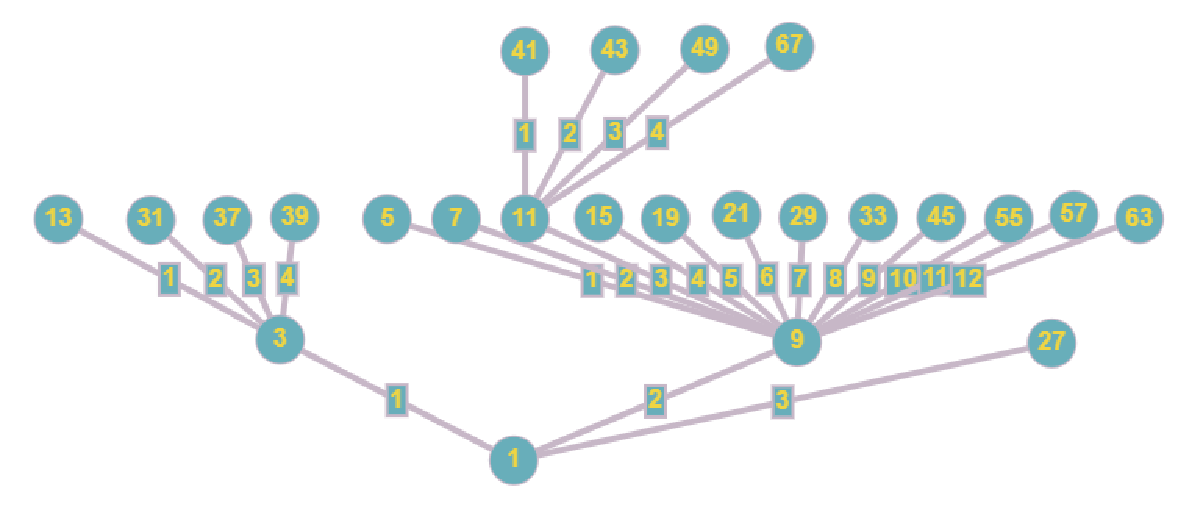}
  \caption{The three levels of $\mathcal{T}_{3,3}^4$.}  \label{fig:3-3-4}
\end{figure}

\subsubsection{General Case}
We now generalize our previous results for a general base $b>1$ and power $e>1$.

\begin{theorem}\label{thm:children}
Fix base $b>1$, power $e>1$, and maximum  number of digits to be $k$. 
For $1\leq i\leq b-1$, let 
$j_i$ denote the number of digits of $m$ equal to $i$. 
Furthermore, 
let $s_i=\sum_{x=i+1}^{b-1} j_x \cdot x^e$, $t_i=\sum_{x=i+1}^{b-1} j_x$, and $u_i=\lfloor\frac{m-s_i}{i^e}\rfloor$.

Then the number of children of $m$ in $\mathcal{T}_{e,b}^k$ is 
\[\sum_{j_2=0}^{u_2}\sum_{j_3=0}^{u_3}\cdots
\sum_{j_{b-2}=0}^{u_{b-2}}
\sum_{j_{b-1}=0}^{u_{b-1}}\binom{k}{j_{b-1}}\binom{k-t_{b-2}}{j_{b-2}}\binom{k-t_{b-3}}{j_{b-3}}\cdots\binom{k-t_{2}}{j_2}\binom{k-t_{1}}{m-s_1}.\]
\end{theorem}
\begin{proof}
The proof is analogous to the proof of Proposition \ref{prop:T23}, by first considering a choice for where to place $j_{b-1}$ digits of $b-1$, which can be done in $\binom{k}{j_{b-1}}$ ways. Then place $j_{b-2}$ digits being $b-2$, which can be done in $\binom{k-j_{b-1}}{j_{b-2}}=\binom{k-t_{b-2}}{j_{b-2}}$ ways. Continuing in this way, we place  $j_i$ digits being $i$, which can be done in $\binom{k-\sum_{x=i+1}^{b-1}j_x}{j_i}=\binom{k-t_{i}}{j_i}$ ways. 
We iterate this process down to placing $j_2$ digits that are twos, which can be done in $\binom{k-\sum_{x=3}^{b-1}j_x}{j_2}=\binom{k-t_2}{j_2}$ ways. 
Thus there must be $m-\sum_{x=2}^{b-1}j_x\cdot x^e=m-s_1$ digits being 1, which can be placed in $\binom{k-\sum_{x=2}^{b-1}j_x}{m-s_1}=\binom{k-t_1}{m-s_1}$ ways. Lastly, any remaining available digits would be uniquely determined to be zeros.

Based on placing the digits from largest to smallest, we note that, for any $1\leq i\leq b-1$, we have that $0\leq j_i\leq \lfloor\frac{m-\sum_{x=i+1}^{b-1}j_x\cdot x^e}{i^e}\rfloor=u_i$.
The result follows from taking the product of all of these choices along with the sum over all possible ranges for the $j_i$'s for all $1\leq i\leq b-1$.
\end{proof}

\section{Unhappy numbers and trees with cycles}\label{sec:cyclicnumbers}

  Recall that unhappy numbers will have a root that is not equal to $1$. That root may be more than one number and so will form a cycle in the graph. In other words, numbers that are not happy can be realized as a graph with one cycle, denoted $\mathcal{C}_{e,b}$ where $e$ is the exponent and $b$ the base as usual.  In the notation $\mathcal{C}^k_{e,b}$, the positive integer $k$ is used to indicate the number of digits in the base $b$ expansion of the numbers displayed in the graph. Additionally, $\mathcal{C}^k_{e,b}[r_0, r_1,\ldots, r_i]$ will be used to mark the cycle of the graph $\mathcal{C}^k_{e,b}$. Here $r_0, r_1,\ldots, r_i$ represent the vertices in the cycle, $r_0$ being the smallest number in base 10 representation and $r_j=S_{e,b}(r_{j-1})$ for each $1\leq j\leq i$.
  Figure \ref{fig:4-4}, at the end of this section, illustrates the graph $\mathcal{C}_{2,7}^3[2,4,16,8]$. 
 
\begin{remark}
  Recall, from Section 1, that the height of an unhappy number is the smallest number of iterations it takes to reach a vertex in the cycle.
    The procedure described in Section \ref{sec:baseb} for constructing happy numbers of a given height, is not restricted only to happy numbers. In fact, the same technique can be used to construct unhappy children of a given height. Moreover, the formula given in Theorem \ref{thm:children} extends naturally to the case where $m$ is an $e$-power $b$-unhappy number.
\end{remark} 

In the remainder of this section, we describe some ideas for additional directions and highlight how to determine the structure of the possible root cycles for unhappy number graphs. 
 
  \subsection{Cycles of length 1}
  
  If a number $X>1$ is equal to the sum of $e$ powers of its own digits in base $b$ then we have a tree with root $X$ and $X$ is not a happy number. For base $b=3$ and exponent $e=2$, there is a tree with root $5$, since $5=(12)_3$ and $1^2+2^2=5$.

  In fact for every odd base $b=2t+1$ for $t>0$ there is such a tree: Let  $X=(x_1x_0)_b$ where $ x_1=t, x_0 = t+1$ for $b=2t+1$ It follows that
  \begin{align} \label{eqn-for-cycle-of-length-1} x_0^2+x_1^2=x_1b+x_0=X
  \end{align}
  proving  that $X$ is the root of a tree for base $b=2t+1,e=2$ and $X$ is not a happy number. This is the first member of the one parameter family of examples \begin{align} \label{family-for-cycle-of-length-1}  
 x_1=t, x_0 = \ell t+1,  b=(\ell ^2+1)t+\ell    
  \end{align}
    that satisfy  (\ref{eqn-for-cycle-of-length-1}) for $\ell >0,t>0$. For $t=1,\ell =2$ and $t=3,\ell =1$ we get the following trees in Figure~\ref{fig:4-1} with $b=7$.

    \begin{figure}[h]
  \centering
  \includegraphics[width=15cm]{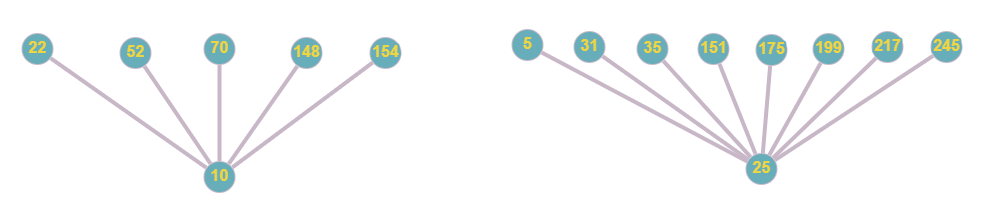}
  \caption{First level of the finite graphs $\mathcal{C}^3_{2, 7}[10]$ and $\mathcal{C}^3_{2, 7}[25]$.}  \label{fig:4-1}
\end{figure}

\subsection{Cycles of length 2}
  If the cycle length is 2 then there are two numbers in base $b$ that generate one another under the happy function. 
  \begin{example}
     In base $b=8$ and with exponent $e=2$, the numbers $13=(15)_8$ and $26=(32)_8$ generate one another because  $1^2+5^2=26$ and $3^2+2^2=13$. We see the first level of this graph in Figure~\ref{fig:4-2}.
  \end{example} 
  
  A formula that generates such pairs for base $b \equiv 3$ mod $5$ goes as follows: Let the numbers be $X$ and $Y$ with base $b$ representations $X=(x_1x_0)_b$ and  $Y=(y_1y_0)_b$. Then choose $x_1=t,x_0=3t+2$ and $y_1=2t+1,y_0=t+1$ with base $b=5t+3$ for $t\ge 0$. It follows that
  \begin{align*}
      x_1^2+x_0^2=y_1b+y_0=Y \ \ \mbox{and} \ \ y_1^2+y_0^2=x_1b+x_0=X.
  \end{align*}
Also note that $Y=2X$.

\begin{figure}[h]  
  \centering
\includegraphics[width=15cm]{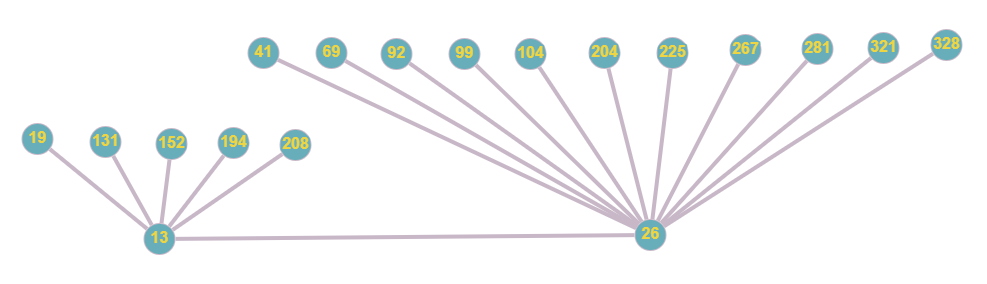}
\caption{First level of the finite graph $\mathcal{C}_{2,8}^3[13,26]$.}  \label{fig:4-2}
\end{figure}

\subsection{Cycles of length 3 or greater}
For a given base $b$ and exponent $e$, there can be more than one cycle of the same length. 
  As an example, when the base $b=11$ and $e=2$, we have root cycles  
   $[74,100,82]$ and $[5,25,13]$. The latter root cycle is seen in Figure~\ref{fig:4-3}.

\begin{figure}[h] 
\centering 
\hspace{1.5in}
\includegraphics[width=8cm]{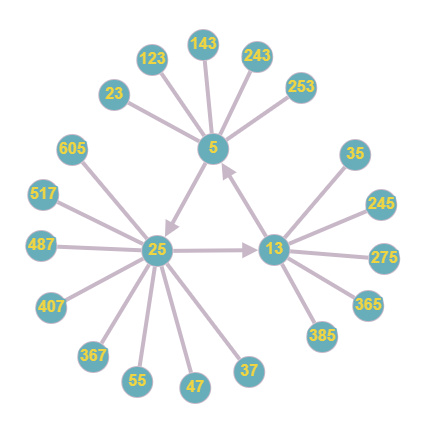}
\caption{First level of the finite graph $\mathcal{C}_{2,11}^3[5,25,13]$.}   \label{fig:4-3}
\end{figure}

In order to generate a one parameter family of such examples one has to solve the nonlinear system of equations
\begin{align*}
  \left[\begin{array}{c}
y_{0}^{2}+y_{1}^{2} 
\\
 z_{0}^{2}+z_{1}^{2} 
\\
 x_{0}^{2}+x_{1}^{2} 
\end{array}\right]=
  \left[\begin{array}{c}
b x_{1}+x_{0} 
\\
 b y_{1}+y_{0} 
\\
 b z_{1}+z_{0} 
\end{array}\right]
\end{align*}
where   $X=(x_1x_0)_b,Y=(y_1y_0)_b$ and  $Z=(z_1z_0)_b$.

 Figure \ref{fig:4-4} shows a cycle of length 4 up to height one for base 7 with power 2 using 3 digits. 

\begin{figure}[h]  
\centering 
\hspace{1.5in}
\includegraphics[width=8cm]{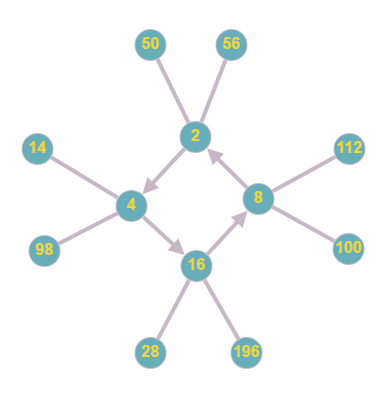}
 \caption{First level of the finite graph $\mathcal{C}_{2,7}^3[2,4,16,8]$.}  \label{fig:4-4}
\end{figure}

\section*{Acknowledgements}
The authors would like to thank the reviewers for their careful reading of our manuscript and for their comments which have helped us improve the final article. The authors thank the Department of Mathematical Sciences at UW Milwaukee for funding which made this research possible. E.~Goedhart also thanks the Association for Women in Mathematics for travel support through the Mathematical Endeavors Revitalization Program. P.~E.~Harris was supported by a Karen Uhlenbeck EDGE Fellowship.

 \end{document}